\documentclass[11 pt]{amsart}

\usepackage[english]{babel}
\usepackage{amssymb} 
\usepackage{amsfonts} 
\usepackage{amsmath}
\usepackage{amsthm} 
\usepackage{epsfig} 
\usepackage{color}
\usepackage{pinlabel}
\usepackage{amscd}
\usepackage[all]{xy}
\usepackage{hyperref}
\usepackage{float}
\usepackage{graphicx}

\newtheorem{lemma}{Lemma}[section]
\newtheorem{thm}[lemma]{Theorem}
\newtheorem{prop}[lemma]{Proposition}
\newtheorem{cor}[lemma]{Corollary}

\newtheorem{fact}[lemma]{Fact}

\theoremstyle{definition}
\newtheorem{defn}[lemma]{Definition}
\newtheorem{quest}[lemma]{Question}

\theoremstyle{remark}
\newtheorem{rem}[lemma]{Remark} 

\newcommand\setItemnumber[1]{\setcounter{enumi}{\numexpr#1-1\relax}}

\newcommand{\Iso}{{\mathrm{Isom}}}
\newcommand{\Vol} {\ensuremath {{\mathrm{Vol}}}}

\newcommand{\matN}{\ensuremath {\mathbb{N}}}
\newcommand{\matR} {\ensuremath {\mathbb{R}}}
\newcommand{\matQ} {\ensuremath {\mathbb{Q}}}
\newcommand{\matZ} {\ensuremath {\mathbb{Z}}}

\newcommand{\matH} {\ensuremath {\mathbb{H}}}

\newcommand{\calT} {\ensuremath {\mathcal{T}}}

\usepackage[letterpaper,left=0.9in,right=0.9in,top=1.1in,bottom=1.1in]{geometry}
\author{Alexander Kolpakov}
\address{Institut de math\'ematiques, Rue Emile-Argand 11, 2000 Neuch\^atel, Switzerland}
\email{kolpakov (dot) alexander (at) gmail (dot) com}

\author{Stefano Riolo}
\address{Institut de math\'ematiques, Rue Emile-Argand 11, 2000 Neuch\^atel, Switzerland}
\email{stefano (dot) riolo (at) unine (dot) ch}

\title[Counting cusped hyperbolic 3--manifolds that bound geometrically]{Counting cusped hyperbolic 3--manifolds \\ that bound geometrically}

\begin{document}

\begin{abstract}
We show that the number of isometry classes of cusped hyperbolic $3$-manifolds that bound geometrically grows at least super-exponentially with their volume, both in the arithmetic and non-arithmetic settings.\\


\noindent
\textit{2010 AMS Classification: } 57R90, 57M50, 20F55, 37F20. \\
\end{abstract}

\maketitle

\section{Introduction}

A \emph{hyperbolic manifold} is a manifold endowed with a Riemannian metric of constant sectional curvature $-1$. Throughout the paper, hyperbolic ma\-ni\-folds are assumed to be connected, orientable, complete, and of finite volume, unless otherwise stated. We refer to \cite{MaRe, VS} for the definition of an arithmetic hyperbolic manifold.

In view of the classical Rokhlin's theorem \cite{GM}, stating that every closed orientable $3$-manifold bounds a compact orientable $4$-manifold, one may translate the question of bounding to the setting of hyperbolic geometry. Following Long and Reid \cite{LR, LR2}, we say that a hyperbolic $n$-manifold $M$ \emph{bounds geometrically} if it is isometric to $\partial W$, for a hyperbolic $(n+1)$-manifold $W$ with totally geodesic boundary. This class of hyperbolic manifolds has attracted significant interest, and for $n=3$ some progress has been recently done -- see \cite{KMT, KRR, MZ, M, S, S2}.

As follows from \cite{KRR, LR}, to bound geometrically is an extremely non-trivial property for a hyperbolic 3-manifold, both in the compact and non-compact setting, respectively (c.f. \cite[Remark 1.4]{KRR}). Despite this, in the present work we show that there are plenty of geometrically bounding cusped (i.e. non-compact) hyperbolic $3$-manifolds. More precisely, the goal of this paper is to show the following

\begin{thm}\label{thm:main}
There exist constants $C,K,v_0,s_0>0$ and a family $\{W_k\}_{k\in\matN}$ of hyperbolic $4$-ma\-ni\-folds with connected, non-compact, totally geodesic boundary $\partial W_k=M_k$ such that:
\begin{enumerate}
\item the $M_k$'s are pairwise non-homeomorphic,
\item the $M_k$'s are all arithmetic (resp. non-arithmetic),
\item the cardinality of $\{k\ |\ \Vol(M_k)\leq v\}$ is at least $v^{Cv}$ for all $v\geq v_0$, 
\item the $M_k$'s have no closed geodesics of length $<s_0$.
\end{enumerate}
Moreover:
\begin{enumerate}
\setItemnumber{5}
\item the doubles $D(W_k)$'s are pairwise non-homeomorphic,
\setItemnumber{6}
\item the $D(W_k)$'s are all arithmetic (resp. non-arithmetic),
\setItemnumber{7}
\item $\Vol(W_k)= K\cdot\Vol(M_k)$.
\end{enumerate}
\end{thm}

Here, the \emph{double} of a hyperbolic manifold $W$ with totally geodesic boundary is the hyperbolic manifold $D(W)=\left(W\sqcup W\right)/_{\partial W}$ obtained by ``reflecting'' $W$ in its boundary $\partial W$.

By combining Theorem~\ref{thm:main} (1)-(4) with the work of Belolipetsky, Gelander, Lubotzky and Shalev \cite{BGLS} and Burger, Gelander, Lubotzky and Mozes \cite{BGLM} on the growth of the number of isometry classes of hyperbolic manifolds, we obtain 

\begin{cor} \label{cor:main_3d}
The number of non-compact, geometrically bounding \emph{arithmetic} hyperbolic $3$-ma\-ni\-folds with volume $\leq v$ has the same growth type as the number of all arithmetic hyperbolic 3-manifolds with volume $\leq v$.

Similarly, the number of non-compact, geometrically bounding \emph{non-arithmetic} hyperbolic $3$-ma\-ni\-folds with volume $\leq v$ and no closed geodesic of length $< s_0$ has the same growth type as the number of all hyperbolic 3-manifolds with volume $\leq v$ and no closed geodesic of length $<s_0$.
\end{cor}

Here, we say that two non-decreasing functions $f,g\colon\matR\to\matR$ have \emph{same growth type} if there exists a constant $C>1$ such that, for all $x$ big enough,

\begin{equation*}
f(x/C)\leq g(x)\leq f(Cx).
\end{equation*}

Indeed, the works \cite{BGLS, BGLM} show that the number of arithmetic hyperbolic $n$-manifolds, $n\geq2$, with volume $\leq v$ has growth type of $v\mapsto v^v$, both in the compact and non-compact case. An analogous statement holds for non-arithmetic hyperbolic $n$-manifolds, $n\geq3$, with an extra lower bound on the length of closed geodesics if $n=3$ (otherwise, by hyperbolic Dehn filling, there are infinitely many $3$-manifolds of volume $\leq v$, provided $v$ is big enough).

Combining the pioneering result by Long and Reid \cite{LR2} with the recent work \cite{KRS} by Reid, Slavich and the first author, many arithmetic hyperbolic manifolds (in arbitrary dimension, compact or non-compact) bound geometrically. The proof there relies on arithmetic techniques and does not provide any counting results on the number of such manifolds with bounded volume. On the other hand, we did not find in the literature any example of a non-arithmetic hyperbolic $3$-manifold which bounds geometrically. In this paper, we provide the first such examples.

Note that Theorem \ref{thm:main} has also some consequences in dimension four. Indeed, by adding also (5)-(7) to \cite{BGLS, BGLM}, we have:

\begin{cor} \label{cor:main_4d}
The number of non-compact, arithmetic (resp. non-arithmetic) hyperbolic $4$-manifolds with volume $\leq v$ and with a connected separating totally geodesic hypersurface  has the same growth type as that of all hyperbolic $4$-manifolds with volume $\leq v$.
\end{cor}

What Theorem \ref{thm:main} actually shows is stronger. Roughly speaking, the statement is: ``Many cusped arithmetic (resp. non-arithmetic) hyperbolic $4$-manifolds are doubles of manifolds with arithmetic (resp. non-arithmetic) connected boundary''. Let us formulate it more precisely in terms of lattices. 

Recall that the orientation-preserving isometry group of the hyperbolic space $\matH^n$ can be identified with $\mathrm{PSO}(n,1)$, and that a \emph{lattice} (resp. \emph{uniform lattice}) is a discrete subgroup $\Gamma<\mathrm{PSO}(n,1)$ such that the quotient $\matH^n/_\Gamma$ has finite volume (resp. is compact). 

\begin{cor} \label{cor:main_4d_bis}
Let $\delta(v)$ be the number of conjugacy classes of torsion-free, non-uniform, arithmetic (resp. non-arithmetic) lattices $\Gamma<\mathrm{PSO}(4,1)$ with co-volume $\leq v$ such that 
\begin{equation*}
\Gamma\cong\Lambda\ast_\Delta\Lambda \mbox{ (amalgamated  free product),} 
\end{equation*}
for some non-uniform, arithmetic (resp. non-arithmetic) lattice $\Delta<\mathrm{PSO}(3,1)$.
Then, $\delta(v)$ has the same growth type as the number of conjugacy classes of all lattices in $\mathrm{PSO}(4,1)$ with co-volume $\leq v$.
\end{cor}

We do not know if the non-arithmetic $D(W_k)$'s of Theorem \ref{thm:main} are commensurable with Gromov -- Piatetsky-Shapiro manifolds \cite{GPS}. However, we do not see any obstruction to this, c.f. \cite[Section 6.2]{FLMS} and Fact \ref{fact:gps} in Section \ref{sec:main_na}.

The proof of Theorem \ref{thm:main} follows by an explicit construction. In both cases, arithmetic and not, the manifolds $M_k$ cover a certain right-angled polyhedron, whose geometric and combinatorial properties help us establishing the necessary facts. 

In the arithmetic case, c.f. Section \ref{sec:main_a}, our construction is an application of a theorem by Martelli \cite{M}, which shows that hyperbolic $3$-manifolds tessellated by right-angled octahedra can be realised as totally geodesic hypersurfaces in hyperbolic $4$-manifolds. The key observation is that the right-angled octahedron is a facet of the right-angled $24$-cell. Some more work is necessary to arrange so that many ``octahedral manifolds'' also bound geometrically, having a fixed-point-free orientation-reversing isometric involution.  We assemble copies of a $4$-dimensional building block (tessellated by right-angled $24$-cells) as prescribed by certain glueing graphs. There exists a one-to-one correspondence between the isometry classes of these manifolds and the isomorphism classes of their glueing graphs. The number of such graphs grows super-exponentially in the number of vertices, and this essentially shows Theorem~\ref{thm:main} (3).

The proof of Theorem \ref{thm:main} in the non-arithmetic case, c.f. Section \ref{sec:main_na}, is similar, but more involved. The main ingredient is the following
\begin{thm}[Proposition \ref{prop:right-angled-na}] \label{thm:P}
There exists a non-compact finite-volume right-angled $4$-polytope $P^4\subset\matH^4$ with a facet $P^3$ such that their respective reflection groups $\Gamma_{P^4}<\mathrm{PO}(4,1)$ and $\Gamma_{P^3}<\mathrm{PO}(3,1)$ are non-arithmetic lattices.
\end{thm}

Then, we proceed in analogy to the arithmetic case, the only difference being in the way of distinguishing manifolds by distinguishing their glueing graphs. More precisely, in the arithmetic case we borrow the techniques of \cite{CFMP} and \cite{KMT} using the Epstein-Penner decomposition \cite{EP}, while in the non-arithmetic case we use a classical result by Margulis \cite{Margulis} on the discreteness of commensurator. For some more observations about the differences between the two cases c.f. Remarks \ref{rem:avoid_margulis} and \ref{rem:margulis_cpt} at the end of the paper.

It is known that hyperbolic right-angled, resp. Coxeter, polytopes of finite volume do not exist in dimensions greater than $12$ \cite{Dufour}, resp. $995$ \cite{Kh, Prokhorov}, so our technique cannot be extended to arbitrary dimensions, c.f. Remarks \ref{rem:surfaces} and \ref{rem:higher_dimension_PV}.
On the other hand,  an advantage of the present construction is that it allows us to prove the bound from Theorem~\ref{thm:main} (7) on the volume of the ``ambient'' $4$-manifold and, in particular, to show Corollaries \ref{cor:main_4d} and \ref{cor:main_4d_bis}. We have not found a simpler way to prove them.

We conclude the introduction by asking the following
\begin{quest} \label{quest:infinitely many}
Are there infinitely many geometrically bounding hyperbolic $3$-manifolds with uniformly bounded volume?
\end{quest}

\section*{Acknowledgements}
\noindent
{\small 
The authors were supported by the Swiss National Science Foundation, project no.~PP00P2-170560. They would like to thank Ruth Kellerhals (Universit\'e de Fribourg), Bruno Martelli (Universit\`a di Pisa), Bram Petri (HIM Bonn), Alan Reid (Rice University), and Leone Slavich (Universit\`a di Pisa) for their interest and fruitful discussions. The authors express special gratitude to the anonymous referee for her/his careful reading of the manuscript and for pointing out a minor mistake in the initial proofs of Lemmas \ref{lem:involution_Ba} and \ref{lem:involution_Bna}.
}

\section{The arithmetic case}\label{sec:main_a}

In this section we prove Theorem~\ref{thm:main} in the arithmetic setting. However, the techniques that we use will not require any specific arithmetic tools to be involved.

\subsection{The block} \label{sec:blocks_a}

Consider an ideal right-angled octahedron $O\subset\matH^3$, and colour its faces in black and white in the chequerboard fashion. By doubling $O$ along the white faces, we produce a hyperbolic $3$-manifold $B'$ with non-compact totally geodesic boundary, colloquially known as the ``Minsky block''.

In total, the manifold $B'$ has $4$ boundary components $D_1,\ldots,D_4$ and 6 annular cusps $C_{ij}$, $i<j\in\{1,\ldots,4\}$, with $C_{ij}\cap\partial B'\subset D_i\cup D_j$. The symmetry group of $B'$ acts transitively on each of the sets $\{D_i\}_i$ and $\{C_{ij}\}_{i<j}$, c.f. \cite{CFMP} or \cite[Proposition 2.3]{S}.

\begin{figure}[h]
\includegraphics[scale=.25]{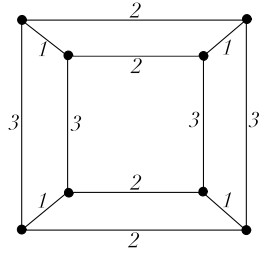}
\caption{\footnotesize This labelled graph will be employed in more than one construction in the present paper.}
\label{fig:cubical_graph}
\end{figure}

Consider now the cubical graph in Figure \ref{fig:cubical_graph}. Take a copy $B'_v$ of $B'$ for each vertex $v$ of the graph, and for each of its edges with label $i\in\{1,2,3\}$ identify the corresponding copies of $D_i$ through the map induced by the identity. One can easily check that the resulting manifold is orientable. 

\begin{defn} \label{def:block_a}
We call the hyperbolic $3$-manifold $B$ with totally geodesic boundary resulting from the above construction \textit{the block}.
\end{defn}

The following fact will be essential later.

\begin{lemma} \label{lem:involution_Ba}
The block $B$ has $8$ mutually isometric boundary components $C_1,\ldots,C_4$, $C'_1,\ldots,C'_4$, and admits an orientation-reversing fixed point free involution $\iota\in \mathrm{Isom}(B)$, such that $\iota(C_j)=C'_j$ for all $j\in\{1,2,3,4\}$.
\end{lemma}
\begin{proof}
By interpreting the graph in Figure \ref{fig:cubical_graph} as the $1$-skeleton of the cube $Q=[-1,1]^3\subset\mathbb R^3$, the antipodal symmetry $a = -\mathrm{id}\in\mathrm{Isom}(Q)$ induces an involution $\iota\in \mathrm{Isom}(B)$. We have $a = r_1 r_2 r_3$, where $r_i\in\mathrm{Isom}(Q)$ is the reflection $x_i\mapsto-x_i$. Each $r_i$ induces an orientation-reversing isometry of $B$ since it exchanges two copies of $B'$ inside $B$ glued along a totally geodesic surface (resulting from the construction of $B$) and acting as a reflection in such a surface. Therefore $\iota$ is orientation-reversing. Moreover the intersection of all the copies of $B'$ in $B$ is empty. This implies that $\iota$ has no fixed point. Finally, since the boundary components $C_i$, $C'_j$, $i, j \in \{ 1, 2, 3, 4\}$ of the block $B$ correspond to the vertices of the cubical graph permuted by $a$, they can be partitioned into pairs $C_i$, $C'_i = \iota(C_i)$, as desired.
\end{proof}

Next, we shall describe the cusps of $B$. The tessellation of $B$ into copies of the octahedron $O$ induces tessellations of the cusp sections into Euclidean squares.

\begin{lemma}\label{lem:cusps_B}
The connected cusp sections of the block $B$ are
\begin{enumerate}
\item tori $T_{2\times4}$ tessellated by $8$ squares as in Figure \ref{fig:squares} (right), and
\item annuli $A_{2\times2}$ tessellated  by $4$ squares as in Figure \ref{fig:squares} (left).
\end{enumerate}
\end{lemma}

\begin{proof}
The shape of each cusp of $B'$ is an annulus $A_{2\times1}$ tessellated by two squares, as depicted in Figure \ref{fig:squares} (left), with $h = 1$. Recall that there is a cusp $C_{ij}$ in $B'$ for each pair of labels $\{i, j\}$, $i<j\in\{1,2,3,4\}$, and vice versa.

By construction, there is a cusp in $B$ for every simple cycle with alternating labels $i,j\in\{1,2,3\}$ in the cubical graph from Figure \ref{fig:cubical_graph}. Its section, up to a homothety, is tessellated by $k$ copies of the annulus $A_{2\times1}$, where $k$ is the length of the cycle. Since each such cycle has length $4$, we obtain the tori $T_{2\times4}$. 

On the other hand, the cusps of $B'$ labelled $(i,4)$ give rise to the cusps of $B$ whose sections are tessellated by $2$ copies of the annulus $A_{2\times1}$. Moreover, these are exactly the cusps of $B$ that intersect $\partial B$ non-emptily. Therefore, they are the annuli $A_{2\times2}$.
\end{proof}

\begin{figure}[ht]
\includegraphics[scale=.3]{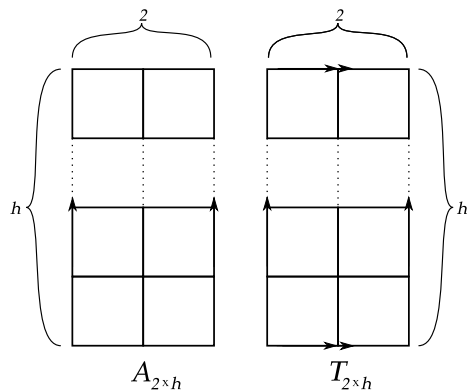}
\caption{\footnotesize The annulus $A_{2\times h}$ (left) and the torus $T_{2\times h}$ (right), with their tessellations into $2h$ copies of the square $Q_0$. Each cusp of the 3-manifold $M_G$ has shape $T_{2\times h}$ for some $h\geq4$.}\label{fig:squares}
\end{figure}

\subsection{Graphs and manifolds} \label{sec:graphs_manifolds_a}

The manifolds $M_k$ from Theorem \ref{thm:main} will be related to graphs of a special kind.

\begin{defn} \label{def:graph}
Let $c$ be a positive integer. A simple $c$-regular graph is called \emph{$1$-factorable} (or, simply, \emph{factorable}), if it admits a proper edge colouring in exactly $c$ colours (i.e. the edges at each vertex are coloured in $c$ colours without repetitions). We shall always suppose that the set of colours is $\{ 1, 2, \ldots, c  \}$. A $c$-regular factorable graph with a chosen colouring is called a \emph{$c$-regular factor}.
\end{defn}

Given a $4$-regular factor $G$, we produce a hyperbolic $3$-manifold. Namely, for each vertex $v$ of $G$, we pick a copy $B^v$ of the block $B$; for every edge $\{v,w\}$ coloured by $j\in\{1,2,3,4\}$, we glue the boundary component $C_j^v$ to $C_j^w$, and ${C'}_j^v$ to ${C'}_j^w$, through the map induced by the identity (c.f. Lemma \ref{lem:involution_Ba} for the notation).

\begin{defn} \label{def:M_G}
For every $4$-regular factor $G$, let $M_G$ be the hyperbolic $3$-manifold described above.
\end{defn}

\begin{rem} \label{rem:arithm}
By construction, there are orbifold coverings $M_G\to B\to B'\to O$ of finite degree -- c.f. Section \ref{sec:colouring_techn} for more details. Moreover, the quotient of $O$ by its symmetry group is isometric to the $[4,3,4]$ hyperbolic orthoscheme, which is an arithmetic orbifold \cite{VS}. In particular, the manifold $M_G$ is arithmetic for any $4$-regular factor $G$.
\end{rem}

It follows from Martelli's work \cite{M}, that the manifold $M_G$ bounds geometrically. 

\begin{prop} \label{prop:bound_a}
The hyperbolic $3$-manifold $M_G$ bounds geometrically a hyperbolic $4$-manifold $W_G$ with
\begin{equation*}
\Vol(W_G)= K\cdot \Vol(M_G),
\end{equation*}
where $K>0$ does not depend on $G$.
\end{prop}
\begin{proof}
Observe that the manifold $M_G$ is tessellated by ideal right-angled octahedra. Moreover, the involution $\iota \in \mathrm{Isom}(B)$ from Lemma \ref{lem:involution_Ba} extends to an orientation-reversing, fixed point free involution $\iota_G \in \mathrm{Isom}(M_G)$. The proof follows by Martelli's argument from \cite[Theorem 3, Corollary 10]{M}, implying that the $4$-manifold $W_G$, tessellated by ideal right-angled 24-cells, can be built explicitly from $M_G$ (c.f. Proposition \ref{prop:bound_na} for more details). 

Compared to our case, Martelli's construction is more general and guarantees the existence of a $W_G$ with
\begin{equation*}
\Vol(W_G)\leq K'\cdot \Vol(M_G).
\end{equation*}
It will be clear later (c.f. the proof of Proposition \ref{prop:bound_na}) that we actually have
\begin{equation*}
\Vol(W_G)= K\cdot \Vol(M_G).
\end{equation*}
\end{proof}

\begin{rem}\label{rem:arithm4d}
Similarly to Remark \ref{rem:arithm}, the double $D(W_G)$ is arithmetic. Indeed, the quotient of the ideal regular the 24-cell by its symmetry group is the arithmetic orthoscheme $[4,3,3,4]$, c.f. \cite{VS}, which by construction is finitely covered by $D(W_G)$.
\end{rem}

In order to distinguish manifolds $M_G$ by distinguishing their ``glueing graphs'' $G$ (c.f. Proposition \ref{prop:distinguish_a} below), we shall proceed similarly to \cite{CFMP} by studying the cusps of $M_G$'s.

In the sequel, $\calT_O$ and $\calT_B$ denote the tessellations of $M_G$ into copies of $O$ and $B$, respectively. Note that the tessellation $\calT_O$ induces a tessellation of the cusp sections of $M_G$ into Euclidean squares.

\begin{lemma}\label{lem:cusps_MG}
A connected cusp section $C$ of the manifold $M_G$ is a torus $T_ {2\times h}$ tessellated by $2h$ squares as in Figure \ref{fig:squares} (right), where:
\begin{enumerate}
\item $h=4$, if $C$ is contained in the interior of a copy of the block $B$ in the tessellation $\mathcal T_B$;
\item $h=2k \geq 8$, otherwise.
\end{enumerate}
\end{lemma}

\begin{proof}
If $C$ is contained in the interior of a copy of $B$ in the tessellation $\mathcal T_B$, then by Lemma \ref{lem:cusps_B} its section is $T_{2\times4}$. Otherwise, $\mathcal T_B$ induces a tessellation of $C$ into annuli $A_{2\times2}$. The number of such annuli equals the length $k$ of the corresponding cycle in the graph $G$ made of edges with alternating colours $i \rightarrow j \rightarrow i \rightarrow \ldots \rightarrow j \rightarrow i$.  But every such a cycle in a $4$-regular factor has length $k \geq 4$.
\end{proof}

\begin{figure}[ht]
\includegraphics[scale=1.5]{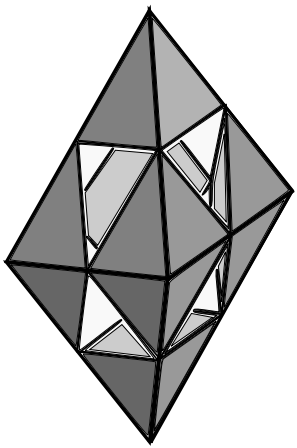}
\caption{\footnotesize The maximal horosection of the octahedron $O$. With the induced Euclidean metric, each square section inside $O$ 
has edge length $1$.}
\label{fig:O-horosection}
\end{figure}

We are ready to show that non-isomorphic $4$-regular factors produce non-isometric manifolds. Here, we consider graphs up to usual graph isomorphism (i.e. not necessarily preserving edge colourings). More precisely:

\begin{prop} \label{prop:distinguish_a}
If two manifolds $M_{G}$ and $M_{G'}$ are isometric, then the graphs $G$ and $G'$ are isomorphic.
\end{prop}

\begin{proof}
Given a manifold $M := M_G$, we show how to recover the glueing graph $G$ solely from the topology and geometry of $M$.

Choose the maximal horosection $S_0$ of the ideal octahedron $O$ depicted in Figure~\ref{fig:O-horosection}, which is preserved by the action of its whole symmetry group. It consists of $6$ copies of a Euclidean square $Q_0$ with edge length $1$, whose interiors are disjoint from each other. Let $S$ denote the cusp section of $M_G$ induced by $S_0$.

By Lemma \ref{lem:cusps_MG}, the set $S\subset M_G$ can be defined as the cusp section of $M_G$ such that each of its connected components has systole $2$ with respect to the induced Euclidean metric. We have recovered the cusp section $S$ from $M_G$. 

By the Epstein -- Penner decomposition theorem \cite{EP}, a cusp section of a non-compact hyperbolic manifold determines a unique (and invariant under isometries) tessellation of the manifold into ideal polytopes, and in the case of $M_G$ the tessellation $\mathcal T_O$ is the one determined by $S$ (and by all the cusp sections of $M_G$ that are homothetic to $S$). This allows us to recover the tessellation $\calT_O$ from $M_G$.

To recover $\calT_B$ from $\calT_{O}$, it suffices to observe that, by Lemma \ref{lem:cusps_MG}, the cusps of $M_G$ with shape $T_{2\times4}$ are exactly those contained in the interior of the copies of $B$ in $\calT_B$. Thus, having recovered $\calT_B$ we also recover the glueing graph $G$.

Finally, any isometry $M_{G}\to M_{G'}$ sends $S$ to $S'$ (the cusp section of $M_{G'}$ whose every connected component has systole $2$), and thus, thanks to the previous discussion, induces a graph isomorphism $G\to G'$.
\end{proof}

If $M_G$ and $M_{G'}$ are non-isometric, then $W_G$ and $W_{G'}$ are non-isometric. This does not imply that the doubles $D(W_G)$ and $D(W_{G'})$ are non-isometric. However, it is not difficult to estimate the number of graphs producing the same $4$-dimensional double.

\begin{prop}\label{prop:distinguish_doubles_a}
Let $G$ be a $4$-factor and $v = \Vol(D(W_G))$. Up to graph isomorphism, there are at most $C\cdot v$ $4$-factors $G'$ such that $D(W_{G'})$ is isometric to $D(W_G)$, where $C > 0$ is a constant independent of $v$ and $G$.
\end{prop}
\begin{proof}
We need an estimate for the number of possible $M_{G'}$'s such that $D(W_{G'})\cong D(W_{G})$. Since there is an isometric involution of $D(W_{G})$ for each such $M_{G'}$, this number does not exceed the cardinality of $\Iso(D(W_{G}))$, which is at most $C\cdot v$ for some $C>0$, as a consequence of the Kazhdan -- Margulis theorem \cite{KM}. Finally, by Proposition \ref{prop:distinguish_a} each such $M_{G'}$ determines $G'$ up to graph isomorphism.
\end{proof}

\subsection{Proof of Theorem \ref{thm:main} in the arithmetic case} \label{sec:conclusion_a}

Let us partition all $4$-regular factors into isomorphism classes (as uncoloured graphs), and choose a representative for each of them.
We index them $G_1, G_2, \ldots$ in a non-decreasing order according to the number of vertices. Let $M_k= M_{G_k}$ and $W_k=W_{G_k}$ denote the corresponding manifolds.

We need the following:

\begin{lemma}\label{lem:graphs}
Let $C(m)$ be the number of $4$-regular factorable graphs with at most $m$ vertices.
Then there exists $C > 1$ such that $C(m)\geq m^{Cm}$, for $m$ big enough.
\end{lemma}   

\begin{proof}
The Lemma follows from the facts that:
\begin{itemize}
\item the number $R(m)$ of all $4$-regular graphs with $\leq m$ vertices has growth type of $m \mapsto m^m$ by \cite{Bollobas}, and
\item the probability $\mathrm{P}(\mbox{G is 4-factorable } | \mbox{ G is 4-regular}) \rightarrow 1$ as $m\to \infty$ by \cite{RW}.
\end{itemize}
\end{proof}

Theorem \ref{thm:main} follows essentially from Proposition \ref{prop:bound_a} and Lemma \ref{lem:graphs}. More precisely:

\begin{enumerate}
\item follows from Proposition \ref{prop:distinguish_a} and the Mostow-Prasad rigidity \cite{Mos,Pra};
\item follows from Remark \ref{rem:arithm};
\item follows from Lemma \ref{lem:graphs} by observing that the number of vertices of the graph $G$ equals the number of copies of the block $B$ in the manifold $M_G$;
\item follows by construction; 
\item follows from Proposition \ref{prop:distinguish_doubles_a} up to taking a subsequence (indeed, $v^v/v$ has growth type of $v^v$);
\item follows from Remark \ref{rem:arithm4d};
\item follows from Proposition \ref{prop:bound_a}.
\end{enumerate}

The proof of Theorem \ref{thm:main} in the arithmetic case is now complete.

\section{The non-arithmetic case}\label{sec:main_na}
In this section, we extend the previous construction to the non-arithmetic case. Beforehand, we will need to build a finite-volume right-angled hyperbolic $4$-polytope with a facet that is a non-arithmetic polyhedron.

\subsection{Coxeter polytopes}\label{sec:coxeter_polytopes}

Recall that a convex polytope $P\subset\matH^n$ is a \emph{Coxeter polytope} if all its dihedral angles are integral submultiples of $\pi$. In this case, the group $\Gamma_P<\Iso(\matH^n)$ generated by reflections through the supporting hyperplanes of $P$ is discrete, and the quotient orbifold $\matH^n/_{\Gamma_P}$ is isometric to $P$. A similar statement holds for Euclidean and spherical Coxeter polytopes. We refer to the book \cite{VS} for all the details about Coxeter polytopes, including the arithmeticity of their reflection groups.

A reflection group $\Gamma_P$ is encoded by its \emph{Coxeter diagram} $\mathcal D$, which is a weighted graph defined as follows. There is a vertex in $\mathcal D$ for each supporting hyperplane of $P$. If two hyperplanes intersect at an angle of $\pi/m$, then the corresponding vertices of $\mathcal D$ are joined by an edge with label $m$. There are two exceptions: if $m=2$ the edge and the label are omitted altogether, while if $m=3$ only the label is omitted. Two hyperplanes that are tangent at infinity (resp. ultraparallel) are represented in $\mathcal D$ by an edge with label $\infty$ (resp. a dashed edge).

The lower-dimensional strata of a Coxeter polytope can also be recovered from its diagram. However, we shall only need to determine which vertices are finite (i.e. located inside $\matH^n$) and which are ideal (i.e. located on the ideal boundary $\partial \matH^n$).

\begin{thm}[Vinberg \cite{V}] \label{thm:vinberg}
Given a Coxeter polytope $P\subset\matH^n$ with Coxeter diagram $\mathcal D$, let $H_1,\ldots,H_k$ be a collection of supporting hyperplanes of $P$ and $\mathcal D'$ be the associated sub-diagram. Then, $H_1\cap\ldots\cap H_k$ is a finite vertex of $P$ if and only if $k=n$ and $\mathcal D'$ is the Coxeter diagram of a spherical reflection group.
\end{thm}

The Coxeter diagrams of spherical and Euclidean reflection groups are classified, c.f. \cite[Chapter 5, $\S$1, Table 1]{VS}. Note that, by the Mostow-Prasad rigidity, in dimension $n\geq3$ the Coxeter diagram of a finite-volume Coxeter polytope $P\subset\matH^n$ determines $P$ up to isometry, and $\Gamma_P$ up to conjugation in $\Iso(\matH^n)$ (c.f. also \cite[\S 3]{Andreev1} and \cite{Andreev2}).

\begin{figure}[ht]
\includegraphics[scale=.85]{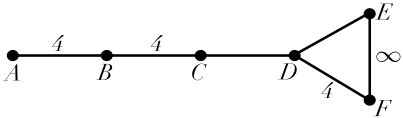}
\caption{\footnotesize The Coxeter diagram for $Q^4$}
\label{fig:q4-diagram}
\end{figure}

\subsection{Two pyramids} \label{sec:pyramids}

The  Coxeter diagram $\mathcal D$ depicted in Figure~\ref{fig:q4-diagram} represents a non-compact polytope $Q^4\subset\matH^4$ of finite volume. Combinatorially, $Q^4$ is a pyramid over a triangular prism, c.f. Tumarkin's list of polytopes in \cite{T}. The Schlegel diagram of $Q^4$ is given in Figure \ref{fig:schlegel}.

\begin{figure}[h]
\includegraphics[scale=.3]{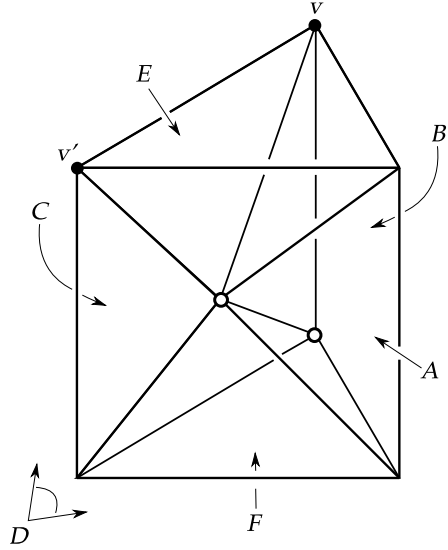}
\caption{\footnotesize A combinatorial picture of the $3$-dimensional boundary of the pyramid $Q^4$, projected into $S^3$. The ideal vertices are hollow. Bold dots indicate the finite vertices $v$ and $v'$, c.f.  Section \ref{sec:P}.}
\label{fig:schlegel}
\end{figure}

Let $H_X\subset\matH^4$ denote the hyperplane associated with the vertex of $\mathcal D$ having label $X\in\{A,\ldots,F\}$, and let $r_X$ be the reflection through $H_X$. Moreover, we identify $\matH^3$ with the hyperplane $H_A\subset\matH^4$. Also, let $Q^3\subset\matH^3$ be the facet $Q^4\cap H_A$ of the pyramid $Q^4$. Observe that combinatorially $Q^3$ is a pyramid over a quadrilateral.

\begin{figure}[h]
\includegraphics[scale=.85]{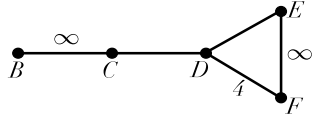}
\caption{\footnotesize The Coxeter diagram for $Q^3$}
\label{fig:q3-diagram}
\end{figure}

\begin{prop}
The facet $Q^3\subset\matH^3$ of $Q^4\subset\matH^4$ is a non-compact finite-volume polyhedron with Coxeter diagram depicted in Figure~\ref{fig:q3-diagram}. Its associated reflection group $\Gamma_{Q^3}<\Iso(\matH^3)$ is non-arithmetic. 
\end{prop}

\begin{proof}
We have to find the dihedral angle in $\matH^3$ between the planes $\matH^3\cap H_X$ and $\matH^3\cap H_Y$, for all $X\neq Y\in\{B,\ldots,F\}$. In this case we shall avoid most of the tedious computations as follows. Observe that $\matH^3=H_A$ is orthogonal to $H_X$ for all $X\in\{C,\ldots,F\}$, and thus the angle between $\matH^3\cap H_X$ and $\matH^3\cap H_Y$ is the same as the angle between $H_X$ and $H_Y$, for all $X\neq Y\in\{C,\ldots,F\}$. These angles can be then transferred from the diagram in Figure~\ref{fig:q4-diagram} to the diagram in Figure~\ref{fig:q3-diagram}.

It remains to compute the angle between $\matH^3\cap H_B$ and $\matH^3\cap H_X$ for $X\in\{C,\ldots,F\}$. As before, since $H_B\perp H_X$ for all $X\in\{D,E,F\}$, then the angle between $\matH^3\cap H_B$ and $\matH^3\cap H_X$ is the same as the angle between $H_B$ and $H_X$ for all $X\in\{D,E,F\}$.

The only dihedral angle left is that between $\matH^3\cap H_B$ and $\matH^3\cap H_C$.  As depicted in Figure \ref{fig:schlegel}, these two planes are tangent at the ideal boundary $\partial \matH^4$. 

The non-arithmeticity of $\Gamma_{Q^3}$ easily follows by applying Vinberg's criterion \cite[Chapter 5, Theorem 3.1]{VS} to the diagram -- see also \cite{GJK}.
\end{proof}

We conclude the section with two more facts about the pyramid $Q^3$. The reader interested only in the proof of Theorem \ref{thm:main} may skip to Section \ref{sec:P}.

Recall that a lattice $\Gamma<\Iso(\matH^n)$ is \emph{maximal} if $\Gamma$ is \emph{not} a proper finite-index subgroup of any discrete group $\Gamma'<\Iso(\matH^n)$. The following observation is non-trivial, and interesting in itself.

\begin{fact} \label{fact:maximal}
The lattice $\Gamma_{Q^3}<\Iso(\matH^3)$ is maximal.
\end{fact}

\begin{proof}
Recall that a cusp of a 3-orbifold $\matH^3/_\Gamma$ is \emph{flexible} if any of its sections covers a Euclidean rectangle.
The pyramid $Q^3=\matH^3/_{\Gamma_{Q^3}}$ is a non-orientable orbifold, and satisfies the following properties:
\begin{enumerate}
\item $\Iso(Q^3)=\{\mathrm{id}\}$, 
\item the cusp of $Q^3$ is flexible, 
\item $\Vol(Q^3)\approx0.40362$.
\end{enumerate}
Property (1) holds because the Coxeter diagram of $\Gamma_{Q^3}$ has no non-trivial symmetry. Property (2) holds because a cusp-section of $Q^3$ is a rectangle. The quantity in (3) is computed, for instance, in \cite[\S 2.3 (2.13)]{GJK}, or can be verified numerically by using Orb \cite{Heard}.

By Adams' \cite[Theorem 4.1]{A}, the volume of an orientable hyperbolic 3-orbifold with a flexible cusp belongs to the set
\begin{equation*}
\{1/4, \sqrt7/8, \sqrt2/4 \}\cup\big(0.3969,+\infty\big).
\end{equation*} 
Since the covering of a flexible cusp is flexible, the volume of a non-orientable hyperbolic 3-orbifold with flexible cusp belongs to the set
\begin{equation*}
\{1/8, \sqrt7/16, \sqrt2/8\}\cup\big(0.19845,+\infty\big).
\end{equation*}

Let us suppose that $\Gamma_{Q^3}$ were not maximal. Then, by (1), we would have $\Gamma_{Q^3}<\Gamma'$ for a lattice $\Gamma'$ with $\left[\Gamma'\colon\Gamma_{Q^3}\right]\geq3$ (otherwise, since index-two subgroups are normal, $Q^3$ would have non-trivial symmetries). Moreover, $\matH^3/_{\Gamma'}$ would be non-orientable, and by (2) its cusp would be flexible.
Then, by (3), we would have
\begin{equation*}
\Vol(\matH^3/_{\Gamma'})=\frac{\Vol(Q^3)}{\left[\Gamma'\colon\Gamma_{Q^3}\right]}\leq\frac{\Vol(Q^3)}{3} \approx 0.1345 < 0.19845.
\end{equation*}

By Adams' theorem, we would actually have that $\left[\Gamma'\colon\Gamma_{Q^3}\right]=3$, because ${\Vol(Q^3)}/4<1/8$.
Again, this is a contradiction, since
\begin{equation*}
0.1345\approx\Vol(\matH^3/_{\Gamma'})\notin\{1/8, \sqrt7/16, \sqrt2/8\}\cup\big(0.19845,+\infty\big),
\end{equation*}
after a simple numerical check. 
\end{proof}

The second observation is that the lattice $\Gamma_{Q^3}$ is a  Gromov -- Piatetsky-Shapiro's hybrid (while we cannot conclude the same for $\Gamma_{Q^4}$):

\begin{fact} \label{fact:gps}
We have
\begin{equation*}
\Gamma_{Q^3}\cong\Lambda_1\ast_\Delta\Lambda_2\quad\mbox{and}\quad\Gamma_{Q^4}\cong\Lambda'_1\ast_{\Delta'}\Lambda'_2,
\end{equation*}
where $\Lambda_i<\Gamma_i$ (resp. $\Lambda'_i<\Gamma'_i$) is generated by reflections through all but one supporting hyperplanes of an arithmetic pyramid with reflection group $\Gamma_i$ (resp. $\Gamma'_i$). Moreover, $\Gamma_1$ and $\Gamma_2$ are incommensurable, while $\Gamma'_1$ and $\Gamma'_2$ are commensurable.
\end{fact}
\begin{proof}
See \cite[Section 2.2 and Section 4]{GJK}.
\end{proof}

\subsection{Two right-angled polytopes} \label{sec:P}

Consider the Coxeter diagram $\mathcal D$ depicted in Figure~\ref{fig:q4-diagram}, which represents the pyramid $Q^4$ introduced in Section \ref{sec:pyramids}. Since the group
\begin{equation*}
\Gamma_v=\langle r_B,r_C,r_D,r_E\rangle<\Gamma_{Q^4}
\end{equation*}
is a spherical reflection group, then $H_B\cap H_C\cap H_D\cap H_E=v\in\matH^4$ is a finite vertex of $Q^4$ by Theorem \ref{thm:vinberg}.

Let us define
\begin{equation*}
P^4=\bigcup_{\gamma\in\Gamma_v}\gamma(Q^4),\quad\mathrm{and}\quad P^3=P^4\cap\matH^3,
\end{equation*}
where the copy of $\matH^3$ inside $\matH^4$ is identified with the hyperplane $H_A$ fixed by the corresponding reflection $r_A$.

\begin{prop} \label{prop:right-angled-na}
The polytopes $P^4\subset\matH^4$ and $P^3\subset\matH^3$ are right-angled, and $P^3$ is a facet of $P^4$. The associated reflection groups $\Gamma_{P^4}<\Iso(\matH^4)$ and $\Gamma_{P^3}<\Iso(\matH^3)$ are non-arithmetic.
\end{prop}
\begin{proof}
Observe that each $2$-face of $Q^4$, or of any of its copies in $P^4$, incident to the vertex $v$ will not be contained in any 2-face of $P^4$.
All the remaining $2$-faces of $Q^4$ carry a dihedral angle of $\pi/4$ or $\pi/2$. Thus, by combining together the copies of $Q^4$ under the action of $\Gamma_v$, each of these angles will be doubled. This implies that the polytope $P^4$ is right-angled. Since $P^3$ is a facet of $P^4$, then $P^3$ is also right-angled, which easily follows by linear algebra.

Since the reflection group $\Gamma_{Q^3}$ is non-arithmetic, so is its finite-index subgroup $\Gamma_{P^3}$. Since $P^3$ is a facet of $P^4$, which is right-angled, then $\Gamma_{P^3}$ embeds into $\Gamma_{P^4}$. Thus the latter is also non-arithmetic. 
\end{proof}

\begin{figure}
\includegraphics[scale=.06]{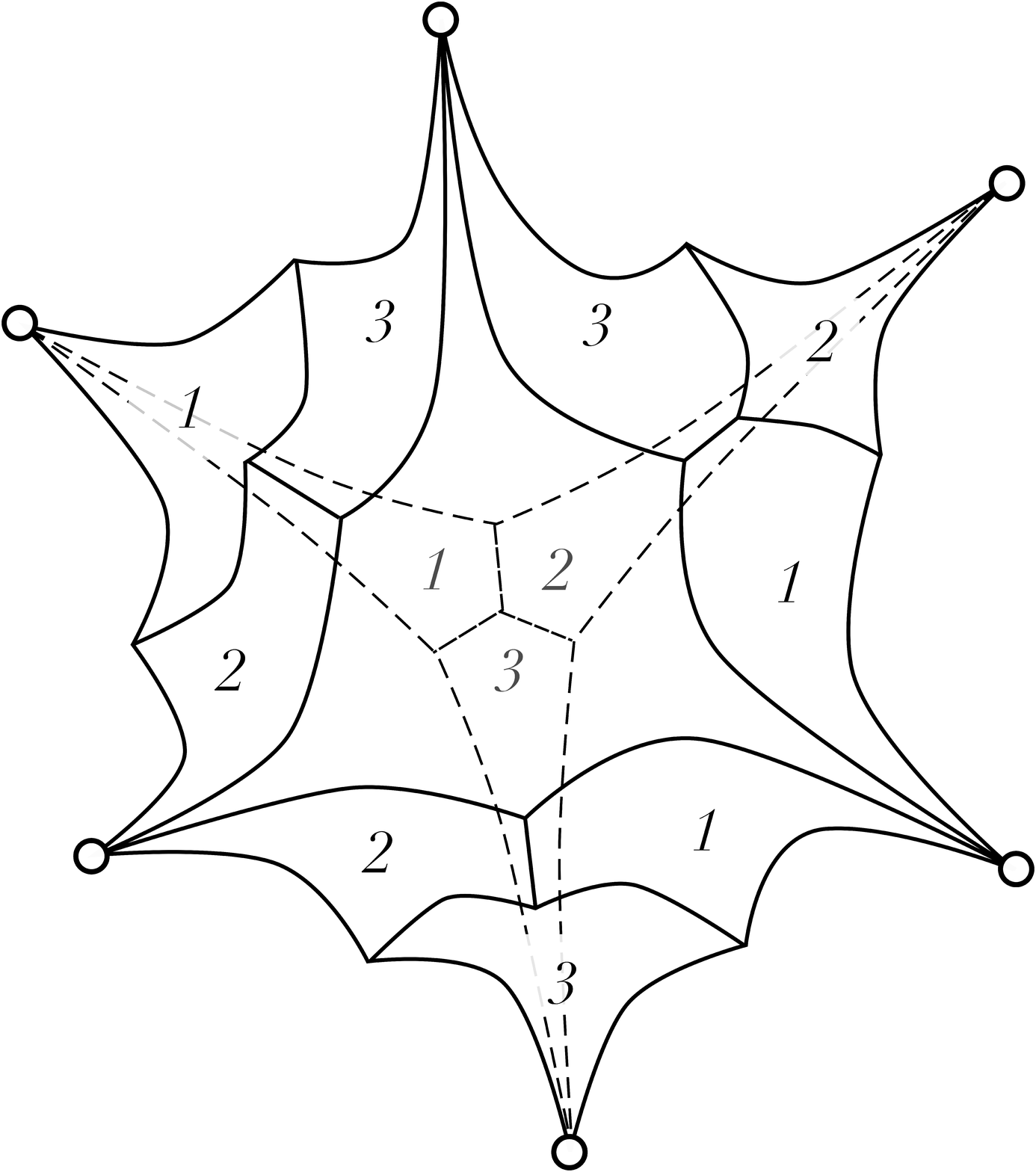}
\caption{\footnotesize The right-angled polyhedron $P^3$. The ideal vertices are represented by white dots. The symmetry group of $P^3$ is that of a regular tetrahedron. The labels $1$, $2$ and $3$ indicate a chosen colouring of the quadrilateral facets of $P^3$ -- see Sections \ref{sec:colouring_techn} and \ref{sec:blocks_na}.}
\label{fig:P^3}
\end{figure}

Let us now describe the combinatorics of the polyhedron $P^3$.

\begin{prop} \label{prop:combinatorics_P^3}
The polyhedron $P^3 \subset \matH^3$ is depicted in Figure \ref{fig:P^3}. In particular, $P^3$ has $4$ pairwise disjoint hexagonal faces (that are pairwise tangent at infinity) and $12$ quadrilateral faces intersecting in triples at its finite vertices. The symmetry group of $P^3$ acts transitively on the set of its hexagonal (resp. quadrilateral) faces. 
\end{prop}

\begin{proof}
By an argument analogous to that in the beginning of this section (this time by looking at the Coxeter diagram in Figure~\ref{fig:q3-diagram}), we obtain that $v'=\matH^3\cap H_C\cap H_D\cap H_E$ is a finite vertex of the pyramid $Q^3$. The group
\begin{equation*}
\Gamma_{v'}=\langle r_C,r_D,r_E\rangle
\end{equation*}
is isomorphic to the symmetric group $\mathfrak S_4$. 

Since $\matH^3=H_A$ is orthogonal to $H_C$, $H_D$ and $H_E$, but \textit{not} to $H_B$, we have that
\begin{equation*}
P^3=\bigcup_{\gamma\in\Gamma_{v'}}\gamma(Q^3).
\end{equation*}

Note that $Q^3$ has trivial symmetry group, so that the symmetry group of $P^3$ is exactly $\Gamma_{v'}\cong\mathfrak S_4$. In other words, the polyhedron $P^3$ has tetrahedral symmetry.

We proceed to studying the faces of $P^3$. By construction,  there are only two types of faces of $P^3$ up to isometry: $P^3\cap H_B$ and $P^3\cap H_F$.

The face $Q^3\cap H_B$ of the pyramid $Q^3$ is a triangle. Since $H_B$ is orthogonal to $H_D$ and $H_E$, but \textit{not} to $H_C$, we obtain
\begin{equation*}
P^3\cap H_B=\bigcup_{\gamma\in\langle r_D, r_E\rangle}\gamma(Q^3)\cap H_B.
\end{equation*}
The group $\langle r_D, r_E\rangle$ is dihedral of order $6$. Moreover, the plane $H_F\cap\matH^3$ intersects $H_D\cap\matH^3$ non-orthogonally, and is tangent at infinity to $H_E\cap\matH^3$. Thus, the face $P^3\cap H_B$ is a hexagon with alternating ideal and finite vertices.
There are $\left[\langle r_C, r_D, r_E\rangle:\langle r_D, r_E\rangle\right]=4$ such hexagons in total, and they are adjacent at ideal vertices only.

The face $Q^3\cap H_F$ of the pyramid $Q^3$ is also a triangle. Since $H_F$ is orthogonal to $H_C$, but is \textit{neither} orthogonal to $H_D$ \textit{nor} to $H_E$, we have that
\begin{equation*}
P^3\cap H_F=(Q^3\cup r_C(Q^3))\cap H_F.
\end{equation*}
Since the plane $H_C\cap\matH^3$ intersects $H_D\cap\matH^3$ non-orthogonally and is tangent at infinity to $H_B\cap\matH^3$, the face $P^3\cap H_F$ is a quadrilateral with one ideal vertex.

Finally, the group $\langle r_C, r_D\rangle$ is dihedral of order $6$, so that we have $\left[\langle r_C,r_D\rangle:\langle r_C\rangle\right]=3$ such quadrilateral faces of $P^3$ that meet at a finite vertex of $P^3$. Moreover, there are $\left[\langle r_C, r_D, r_E\rangle:\langle r_C, r_D\rangle\right]$ $=4$ such vertices, and thus $3\cdot 4=12$ quadrilateral faces in total.

All the information we collected is now sufficient to draw Figure \ref{fig:P^3}.
\end{proof}

\subsection{The colouring technique} \label{sec:colouring_techn}
Let us briefly recall a standard technique to produce hyperbolic ma\-ni\-folds from right-angled polytopes, and some of its straightforward generalisations. For more details, we refer the reader to \cite[Sections 1.1 -- 1.2]{M}.

Given a right-angled polytope $P\subset\matH^n$ of finite volume, we assign surjectively to each facet a colour in a set $\{c_1,\ldots,c_k\}$, so that any two intersecting facets have distinct colours (however, if two facets are adjacent only at an ideal vertex, they are allowed to have the same colour).  

Then we take $2^k$ (coloured) copies of $P$, say $P^{\epsilon_1,\ldots,\epsilon_k}$ with $(\epsilon_1,\ldots,\epsilon_k)\in\{0,1\}^k$, and glue through the map induced by the identity each facet of $P^{\epsilon_1,\ldots,\epsilon_k}$ with colour $c_i$ to the corresponding facet of $P^{{\epsilon}_1',\ldots,{\epsilon}_k'}$, for all $(\epsilon_1,\ldots,\epsilon_k)$ and $(\epsilon_1',\ldots,\epsilon_k')$ which differ only in their $i$-th entry.

The resulting metric space is a hyperbolic manifold $M=\matH^n/_\Gamma$.  The group $\Gamma$ is the kernel of an epimorphism $\Gamma_P\to\left(\matZ/_{2\matZ}\right)^k$ defined by the colouring, where $\Gamma_P$ is the reflection group associated with $P$.

It is possible to generalise this construction in more than one way. E.g., if some pairwise disjoint facets of $P$ are left uncoloured, then $M$ is a hyperbolic manifold with totally geodesic boundary. Moreover, one can colour the components of $\partial M$ and apply an analogous procedure to $M$ itself.
Exactly this procedure was performed in Section \ref{sec:blocks_a} to construct the arithmetic blocks $B'$ and $B$.

Let us further generalise the colouring technique. Namely, if some uncoloured facets of $P$ intersect, then $M$ is a \emph{hyperbolic manifold with right-angled corners}, that is, a manifold locally modelled on an orthant of $\matH^n$.  Same as for right-angled polytopes, the boundary of a hyperbolic manifold with right-angled corners is naturally decomposed into \emph{facets}, which in their own turn are hyperbolic manifolds with right-angled corners.
Then the colouring technique can be applied, c.f. see \cite[Proposition 6]{M}.

\subsection{The block} \label{sec:blocks_na}

Let us assign three distinct colours $c_1$, $c_2$ and $c_3$ to the quadrilateral faces of the polyhedron $P^3$, as indicated in Figure \ref{fig:P^3}. We call $B'$ the resulting space obtained by applying the above colouring technique to $P^3$.

The glueing graph of $B'$ is the cubical graph depicted in Figure \ref{fig:cubical_graph}. Each vertex of this graph represents a copy of the polyhedron $P^3$, and two vertices are joined by an edge labelled $j\in\{1,2,3\}$ if the four quadrilateral faces coloured $c_j$ in a copy of $P^3$ are identified with those of the other copy.

\begin{lemma}
The space $B'$ is a hyperbolic manifold with totally geodesic boundary that consists of four $12$-punctured spheres. The symmetry group of $B'$ acts transitively on the set of its boundary components.
\end{lemma}
\begin{proof}
The hexagonal faces of $P^3$ are the only uncoloured faces, and they are pairwise disjoint, so that $B'$ has totally geodesic boundary. By Proposition \ref{prop:combinatorics_P^3}, the symmetry group of $B'$ acts transitively on the set of its boundary components.

For each hexagonal face $H$ of the polyhedron $P^3$, there is a boundary component of $B'$. Indeed,
the $3$-colouring of $P^3$ induces a $3$-colouring of $H$, in which all the three colours are used. The glueing graph of each boundary component of $B'$ is again the cubical graph from Figure \ref{fig:cubical_graph}: now the vertices of the graph represent copies of $H$, and we readily observe that each boundary component of $B'$ is a $12$-punctured sphere (the punctures correspond to the edges of the cubical graph).
\end{proof}

\begin{defn} \label{def:block_na}
We assign three distinct colours to three of the four boundary components of $B'$. We call the hyperbolic $3$-manifold $B$ with totally geodesic boundary obtained by applying the colouring technique to $B'$ \textit{the block}.
\end{defn}
The glueing graph of $B$ is once again the cubical graph from Figure \ref{fig:cubical_graph}, where now the vertices of the graph represent copies of $B'$.

\begin{lemma} \label{lem:involution_Bna}
The block $B$ has $8$ mutually isometric boundary components $C_1,\ldots,C_4$, $C'_1,\ldots,C'_4$, and admits an orientation-reversing  fixed point free involution $\iota\in \mathrm{Isom}(B)$, such that $\iota(C_j)=C'_j$ for all $j\in\{1,2,3,4\}$.
\end{lemma}
\begin{proof}
As in Lemma \ref{lem:involution_Ba}, the antipodal symmetry of the cube induces such an involution $\iota\in\mathrm{Isom}(B)$.
\end{proof}

\subsection{Graphs and manifolds} \label{sec:graphs_manifolds_na}

Exactly as in Section \ref{sec:graphs_manifolds_a}, given a $4$-factor $G$ (see Definition \ref{def:graph}), we build a hyperbolic $3$-manifold $M_G$: this time by taking the block $B$ of Definition \ref{def:block_na}.

\begin{rem} \label{rem:nonarithm}
By construction, there are orbifold coverings $M_G\to B\to P^3\to Q^3$ of finite degree. In particular, the manifold $M_G$ is non-arithmetic.
\end{rem}

The following is an analogue of Proposition \ref{prop:bound_a}. The argument to prove it is very similar to Martelli's \cite[Theorem 3, Corollary 10]{M}, where instead of $P^3$ and $P^4$ one has the ideal regular octahedron and the ideal regular $24$-cell, respectively.

\begin{prop} \label{prop:bound_na}
The hyperbolic $3$-manifold $M_G$ bounds geometrically a hyperbolic $4$-manifold $W_G$, such that
\begin{equation*}
\Vol(W_G)= K\cdot \Vol(M_G),
\end{equation*}
where $K>0$ does not depend on $G$.
\end{prop}

\begin{proof}
Observe that the $3$-manifold $M_G$ is tessellated by copies of the polyhedron $P^3$, which is a facet of the right-angled polytope $P^4$.
Since all the glueing maps are induced by the identity, we can naturally place a copy of $P^4$ ``above'' each copy of $P^3$, to get a hyperbolic $4$-manifold with right-angled corners $W'_G$ (c.f. Section \ref{sec:colouring_techn}).

More precisely, the space $W'_G$ is obtained as follows. Consider the $12$ facets of $P^4$ that intersect $P^3$ in its quadrilateral faces. By colouring them with $3$ colours as depicted in Figure \ref{fig:P^3}, and by applying the colouring technique to $P^4$ (where $P^3$ remains uncoloured, and the neighbouring facets of $P^4$ acquire the colours induced from the respective faces of $P^3$), we obtain a hyperbolic $4$-manifold with right-angled corners $B^4$.
The facets of $B^4$ are partitioned into $3$ sets:
\begin{itemize}
\item the ``bottom'' facet $B^3$ (that is, the block $B$)
\item the facets intersecting $B^3$, called \emph{vertical} facets, and
\item the remaining facets, called \emph{top} facets.
\end{itemize}
Then, $W'_G$ is obtained by glueing copies of $B^4$ along the vertical facets through the identity map as prescribed by the graph $G$. One connected component of $\partial W'_G$ is isometric to $M_G$, and (being totally geodesic) is a facet of $W'_G$.

Let us now consider the remaining facets of $W'_G$. Since $W'_G$ is obtained by glueing copies of $B^4$ along vertical facets through the identity map, we can colour the remaining facets of $W'_G$ with $N$ colours (distinct from the ones already used), where $N$ is the number of top facets of $B^4$. By applying the colouring technique to $W'_G$, we get a hyperbolic $4$-manifold $W''_G$ with totally geodesic boundary $\partial W''_G$ isometric to many disjoint copies of $M_G$. The manifold $W''_G$ is tessellated by $2^N$ copies of $W'_G$.

Observe that $M_G$ has an orientation-reversing fixed point free  isometric involution $\iota_G$. Indeed, $\iota_G$ is obtained by applying the involution $\iota\in \mathrm{Isom}(B)$ from Lemma \ref{lem:involution_Bna} to each copy  of the block $B$ in $M_G$. The map $\iota_G$ is well-defined, since all glueing maps between the copies of $B$ in $M_G$ are induced by the identity. We quotient all but one boundary components of $W''_G$ by the involution $\iota_G$, to get the desired $W_G$.

Finally, we compute
\begin{equation*}
\Vol(W_G)=2^N\cdot\Vol(W'_G)=2^{N}\cdot\sharp\{\mathrm{vertices\ of\ }G\}\cdot\Vol(B^4) = 
\end{equation*}
\begin{equation*}
=\frac{2^N\cdot\Vol(B^4)}{\Vol(B^3)}\Vol(M_G) = K\, \Vol(M_G).
\end{equation*}
\end{proof}

\begin{rem} \label{rem:non-arithm.4d}
Similarly to Remark \ref{rem:arithm4d}, we have that the double $D(W_G)$ is non-arithmetic. Indeed, by construction, $D(W_G)$ finitely covers the non-arithmetic orbifold $Q^4\cong P^4/_{\Iso(P^4)}$ of Section \ref{sec:pyramids}.
\end{rem}

Since the polyhedron $P^3$ is not ideal, we cannot apply the argument of Proposition \ref{prop:distinguish_a} in order to distinguish manifolds by distinguishing their glueing graphs. We shall rather exploit the non-arithmeticity of $P^3$.

The following is a consequence of a classical result by Margulis \cite{Margulis} about commensurators of lattices, stated here in a way that will be convenient for us (c.f. \cite[Theorem 10.3.5]{MaRe} and the discussion in \cite{NR}):

\begin{thm}[Margulis]\label{thm:margulis}
Let $\Gamma<\mathrm{PO}(n,1)$, $n\geq3$, be a non-arithmetic lattice. Then there is a unique maximal lattice $\Gamma_{\mathrm{max}}$ containing $\Gamma$. Moreover, every subgroup $\Gamma'<\Gamma_\mathrm{max}$ isomorphic to $\Gamma$ is conjugate to $\Gamma$ in $\Gamma_\mathrm{max}$.
\end{thm}

We are finally ready to prove the following

\begin{prop} \label{prop:distinguish_na}
Let $G$ be a $4$-factor and $v = \Vol(M_G)$. Then the number of $4$-factors $G'$ such that $M_{G'}$  is isometric to $M_G$ is at most $C\cdot v$, where $C > 0$ is a constant independent of $v$ and $G$.
\end{prop}

\begin{proof}
Since $M_G\cong\matH^3/_\Gamma$, $B\cong\matH^3/_{\Gamma_B}$, by recalling Remark \ref{rem:nonarithm} we can assume that $\Gamma<\Gamma_B$.

Now we show that the factor $G$ is determined by the inclusion $\Gamma<\Gamma_B$. Indeed, the group $\Gamma_B$ has a finite presentation $\mathcal{P}$ with a generator of order $2$ for each boundary component of $B$, and such that all the remaining generators have infinite order. Let $\bar G$ be the Schreier co-set graph of $\Gamma<\Gamma_B$ with respect to $\mathcal{P}$. Then $\bar G$ has its edges coloured by the generators from the presentation $\mathcal{P}$. The factor $G$ is readily obtained from $\bar G$ by deleting all the edges coloured by generators of infinite order.

Now, assuming that $M_G\cong M_{G'}\cong\matH^3/_{\Gamma'}$ and $\Gamma'<\Gamma_B$, we have that $\Gamma$ and $\Gamma'$ are conjugate in $\Iso(\matH^3)$. By Margulis' Theorem \ref{thm:margulis}, they are actually conjugate in $\Gamma_{\mathrm{max}}$. The number of conjugacy classes of $\Gamma$ in $\Gamma_{\mathrm{max}}$ is at most the index $[\Gamma_{\mathrm{max}}:\Gamma]=C\cdot v$, where $C^{-1} = \Vol(\matH^3/_{\Gamma_{\mathrm{max}}}) = \Vol(Q^3)$ (the last equality follows from Fact \ref{fact:maximal}, however it is not necessary).

Thus, the upper bound of $C\cdot v$ holds for the number of all the possible $\Gamma'<\Gamma_B$ with $\matH^3/_{\Gamma'}\cong\matH^3/_{\Gamma}$ and, by the previous discussion, also for the number of all possible $4$-factors $G'$ with $M_{G'}\cong M_{G}$. 
\end{proof}

Similarly to Proposition \ref{prop:distinguish_doubles_a}, we can estimate the number of graph factors producing the same $4$-dimensional double.

\begin{prop} \label{prop:distinguish_doubles_na}
Let $G$ be a $4$-factor and $v = \Vol(D(W_G))$. Then there are at most $C\cdot v^2$ $4$-factors $G'$ such that $D(W_{G'})$  is isometric to $D(W_G)$, where $C > 0$ is a constant independent of $v$ and $G$.
\end{prop}

\begin{proof}
Let us fix $D(W_{G})=\matH^4/_\Gamma$. We want to estimate the number of possible $M_{G'}$'s such that $D(W_{G'})\cong D(W_{G})$. Since there is an isometric involution of $D(W_{G})$ for each such $M_{G'}$, this number does not exceed the cardinality of $\Iso(D(W_{G}))$, which is at most $C_1\cdot v$ for some $C_1>0$, c.f. \cite{KM}.

Now, we fix such an $M_{G'}$. The number of $4$-factors $G''$ such that $M_{G''} \cong M_{G'}$ is at most $C_2\cdot\Vol(M_{G''})=C_3\cdot v$, for some constants $C_2, C_3>0$, where the estimate comes from Proposition \ref{prop:distinguish_na} and the subsequent equality holds by Proposition \ref{prop:bound_na}.

Finally, the statement follows by bringing together the above two facts and putting $C = C_1\cdot C_3$.
\end{proof}

\subsection{Conclusion and remarks} \label{sec:conclusion_na}

The proof of Theorem~\ref{thm:main} in the non-arithmetic case proceeds from now on similarly to the arithmetic one. The difference here is that on one hand we do not need to count $4$-factors up to graph isomorphism, however we need to exclude some of the $M_G$'s from the sequence.

As in Section \ref{sec:conclusion_a}, (2) follows from Remark \ref{rem:arithm}, (4) by construction, (6) from Remark \ref{rem:arithm4d}, and (7) from Proposition \ref{prop:bound_na}.
Now, by Proposition \ref{prop:distinguish_na}, the number of manifolds $M_G$ produced from $4$-factors $G$ with $\leq v/\Vol(B)$ vertices has growth type of $f(v) = v^v$, while the number of non-isomorphic $4$-factors $G^\prime$ producing manifolds $M_{G^\prime}$ isometric to $M_G$ has growth type of $g(v) = v$. However, $f(v)/g(v)$ has growth type of $v^v$ again. In particular, we have proved (1) and (3) up to taking a subsequence. Up to another subsequence, by Proposition \ref{prop:distinguish_doubles_na} also (5) holds, since $v^v/v^2$ also has growth type of $v^v$.

The proof of Theorem~\ref{thm:main} is finally completed. We conclude the paper with a series of remarks.

\begin{rem} \label{rem:surfaces}
One could apply this technique to the right-angled octahedron to show a similar result for geometrically bounding arithmetic surfaces. Moreover, by glueing right-angled polyhedra, one could easily build many explicit examples of arithmetic or non-arithmetic, compact or non-compact, geometrically bounding hyperbolic surfaces.
\end{rem}

\begin{rem} \label{rem:higher_dimension_PV}
This technique can be probably extended to dimensions $4\leq n\leq7$ by means of the Potyagailo-Vinberg right-angled polytopes \cite{PV}. However, some more work would be necessary to distinguish manifolds by distinguishing their glueing graphs in this case. Indeed, the aforementioned polytopes are arithmetic (so one cannot use Margulis' Theorem) and have some finite vertices (so one cannot use the Epstein-Penner polyhedral decomposition either).
\end{rem}

\begin{rem} \label{rem:avoid_margulis}
By glueing copies of another non-arithmetic $4$-polytope studied in \cite{KS, MR}, one can show that there is a hyperbolic $4$-manifold with right-angled corners and a facet that resembles the block $B$ of Definition \ref{def:block_a}. Namely, this facet is a hyperbolic $3$-manifold with totally geodesic boundary, and it covers an \emph{ideal} right-angled icosidodecahedron, which is non-arithmetic. With this in hand, one could prove Theorem \ref{thm:main} without appealing to Margulis' Theorem, and rather using the Epstein-Penner polyhedral decomposition again. However, such an alternative ``arithmetic-free'' proof encounters much technical difficulties: for example, one needs at least $600$ copies of the initial $4$-polytope to begin with. 
\end{rem}

\begin{rem} \label{rem:margulis_cpt}
An advantage of Margulis' Theorem, instead, is that it can be directly applied to \emph{compact} non-arithmetic manifolds. Unfortunately, we do not know any example of a non-arithmetic compact right-angled $4$-polytope (with a non-arithmetic facet). Does such a polytope exist?
\end{rem}


\begin{thebibliography}{100} 
\normalsize

\bibitem{A} \textsc{C.~C.~Adams}: \emph{Noncompact hyperbolic 3-orbifolds of small volume}, Topology '90,  Ohio State Univ. Math. Res. Inst. Publ. \textbf{1}, pp.~1--15, Berlin: de Gruyter (1992).

\bibitem{Andreev1}\textsc{E.~M.~Andreev}: \emph{On convex polyhedra in Loba\u{c}evski\u{\i} space}, Math. USSR Sb. \textbf{10}, no. 3 (1970), 413--440.

\bibitem{Andreev2}\textsc{E.~M.~Andreev}: \emph{On convex polyhedra of finite volume in Loba\u{c}evski\u{\i} space}, Math. USSR Sb. \textbf{12}, no. 2 (1970), 255--259.

\bibitem{BGLS} \textsc{M.~Belolipetsky -- T.~Gelander -- A.~Lubotzky -- A.~Shalev}: \emph{Counting arithmetic lattices and surfaces}, Ann. of Math. (2) \textbf{172}, no. 3 (2010), 2197--2221.

\bibitem{Bollobas}\textsc{B. Bollob\'{a}s}: \emph{The asymptotic number of unlabelled regular graphs}, J. London Math. Soc. \textbf{26}, no. 2 (1982), 201--206.

\bibitem{BGLM} \textsc{M.~Burger -- T.~Gelander -- A.~Lubotzky -- S.~Mozes}: \emph{Counting hyperbolic manifolds}, Geom. Funct. Anal. \textbf{12}, no. 6 (2002), 1161--1173.

\bibitem{CFMP} \textsc{F.~Costantino -- R.~Frigerio -- B.~Martelli -- C.~Petronio}: \emph{Triangulations of 3-manifolds, hyperbolic relative handlebodies, and Dehn filling}, Comment. Math. Helv. \textbf{82}, no.~4 (2007), 903--933.

\bibitem{Dufour} \textsc{G.~Dufour}: \emph{Notes on right-angled Coxeter polyhedra in hyperbolic spaces}, Geom. Dedicata \textbf{147}, no.~1 (2010), 277--282.

\bibitem{EP} \textsc{D.~B.~A.~Epstein -- R.~C.~Penner}: \emph{Euclidean decompositions of non-compact hyperbolic manifolds}, J. Diff. Geom. \textbf{27} (1988), 67--80.

\bibitem{FLMS} \textsc{D.~Fisher -- J.-F.~Lafont -- N.~Miller -- M.~Stover}: \emph{Finiteness of Maximal Geodesic Submanifolds in Hyperbolic Hybrids}, \texttt{arXiv:1802.04619}.

\bibitem{GPS} \textsc{M.~Gromov -- I.~Piatetski-Shapiro}: \emph{Nonarithmetic groups in Lobachevski spaces}, Inst. Hautes \'Etudes Sci. Publ. Math. \textbf{66} (1988), 93-103.

\bibitem{GJK} \textsc{R.~Guglielmetti -- M.~Jacquemet -- R.~Kellerhals}: \emph{Commensurability of hyperbolic Coxeter groups: theory and computation}, RIMS K\^oky\^uroku Bessatsu \textbf{B66} (2017), 57-113.

\bibitem{GM}\textsc{L.~Guillou, A.~Marin, ed.}: \emph{\`A la recherche de la topologie perdue}, Progress in Mathematics \textbf{62}, Basel: Birkh\"{a}user (1986).    

\bibitem{Heard}\textsc{D. Heard}, \emph{Orb, an interactive software for finding hyperbolic structures on 3-dimensional orbifolds and manifolds}, available from \href{https://github.com/DamianHeard/orb}{GitHub}

\bibitem{KM}\textsc{D. A.~Kazhdan -- G. A.~Margulis}: \emph{A proof of Selberg's hypothesis}, Math. Sbornik (N.S.) \textbf{75} (1968), 162--168.

\bibitem{KS} \textsc{S.~Kerckhoff -- P.~Storm}: \emph{From the 24-cell to the cuboctahedron}, Geom. Topol. \textbf{14} (2010), 1383--1477.

\bibitem{Kh}\textsc{A.~G.~Khovanski\u{\i}}: \emph{Hyperplane sections of polyhedra, toric varieties, and discrete groups in Lobachevskij spaces}, Funct. Anal. Appl. \textbf{20} (1986), 41--50.

\bibitem{KMT} \textsc{A.~Kolpakov -- B.~Martelli -- S.~Tschantz}: \emph{Some hyperbolic three-manifolds that bound geometrically}, Proc. Amer. Math. Soc. \textbf{143}, no. 9 (2015), 4103--4111.

\bibitem{KRR} \textsc{A.~Kolpakov -- A.W.~Reid -- S.~Riolo}: \emph{Many cusped hyperbolic 3-manifolds do not bound geometrically}, To appear in Proc. Amer. Math. Soc. \texttt{arXiv:1811.05509}.

\bibitem{KRS} \textsc{A.~Kolpakov -- A.W.~Reid -- L.~Slavich}: \emph{Embedding arithmetic hyperbolic manifolds}, Math. Research Letters \textbf{25} (2018), 1305--1328.

\bibitem{LR} \textsc{D.~D.~Long -- A.~W.~Reid}: \emph{On the geometric boundaries of hyperbolic 4-manifolds}, Geom. Topol. \textbf{4} (2000), 171--178.

\bibitem{LR2} \textsc{D.~D.~Long -- A.~W.~Reid}: \emph{Constructing hyperbolic manifolds which bound geometrically}, Math. Res. Lett. \textbf{8}, no. 4, (2001), 443--455.

\bibitem{MZ} \textsc{J. Ma -- F. Zheng}: \emph{Geometrically bounding 3-manifold, volume and Betti number}, \texttt{arXiv:1704.02889}.

\bibitem{MaRe} \textsc{C.~Maclachlan -- A.~W.~Reid}: \emph{The arithmetic of hyperbolic 3-manifolds}, Graduate Texts in Mathematics \textbf{219}, Springer-Verlag, New York (2003).

\bibitem{Margulis} \textsc{G.~A.~Margulis}: \emph{Discrete subgroups of semisimple Lie groups}, Ergebnisse der Mathematik und ihrer Grenzgebiete \textbf{3}, Springer-Verlag, Berlin (1991).

\bibitem{M} \textsc{B.~Martelli}: \emph{Hyperbolic three-manifolds that embed geodesically}, \texttt{arXiv:1510.06325}.

\bibitem{MR}\textsc{B.~Martelli -- S.~Riolo}: \emph{Hyperbolic Dehn filling in dimension four}, Geom. Topol. \textbf{22}, no. 3 (2018), 1647--1716. 

\bibitem{Mos} \textsc{G.~D.~Mostow}: \emph{Strong rigidity of locally symmetric spaces}, Annals Math. Studies \textbf{78}, Princeton Univ. Press (1973).

\bibitem{NR} \textsc{W.~D.~Neumann -- A.~W.~Reid}: \emph{Arithmetic of hyperbolic manifolds}, Topology ’90,  Ohio State Univ. Math. Res. Inst. Publ., \textbf{1} (1992), 273--310.

\bibitem{PV}\textsc{L.~Potyagailo -- E.~Vinberg}: \emph{On right-angled reflection groups in hyperbolic spaces}, Comment. Math. Helv. \textbf{80} (2005), 1--12.

\bibitem{Pra} \textsc{G.~Prasad}: \emph{Strong rigidity of $\matQ$ rank 1 lattices}, Invent. Math. \textbf{21} (1973), 255--286.

\bibitem{Prokhorov}\textsc{M.~N.~Prokhorov}: \emph{The absence of discrete reflection groups with non- compact fundamental polyhedron of finite volume in Lobachevsky space of large dimension}, Math. USSR Izv. \textbf{28} (1987), 401--411.

\bibitem{RW}\textsc{R.~W.~Robinson -- N.~C.~Wormald}: \emph{Almost all regular graphs are Hamiltonian}, Random Struct. Alg. \textbf{5} (1994), 363--374.

\bibitem{S} \textsc{L.~Slavich}: \emph{A geometrically bounding hyperbolic link complement}, Algebr. Geom. Topol. \textbf{15}, no. 2 (2015), 1175--1197.

\bibitem{S2} \textsc{L.~Slavich}: \emph{The complement of the figure-eight knot geometrically bounds},  Proc. Amer. Math. Soc. \textbf{145}, no. 3 (2017), 1275--1285.

\bibitem{T} \textsc{P.~Tumarkin}: \emph{Hyperbolic Coxeter $n$-polytopes with $n+2$ facets}, \texttt{arXiv:math/0301133}.

\bibitem{V} \textsc{E.~B.~Vinberg}, \emph{Hyperbolic reflection groups}, Russian Math. Surveys \textbf{40} (1985), 31--75.

\bibitem{VS} \textsc{E.~B.~Vinberg -- O.~V.~Shvartsman}: \emph{Discrete groups of motions of spaces of constant curvature}. From: \emph{Geometry II}, Enc. Math. Sci. \textbf{29}, Berlin: Springer (1993) 139--248.

\end{thebibliography}
\end{document}